\documentclass[twoside]{article}

\usepackage[accepted]{aistats2023}
\usepackage{natbib}
\usepackage{amsmath,amssymb,mathrsfs}
\usepackage[utf8]{inputenc} 
\usepackage[T1]{fontenc}    
\usepackage{hyperref}       
\usepackage{url}            
\usepackage{cleveref}
\usepackage{autonum}
\usepackage{booktabs}       
\usepackage{amsfonts}       
\usepackage{nicefrac}       
\usepackage{microtype}      
\usepackage[dvipsnames,svgnames]{xcolor}         
\usepackage{amsthm}

\hypersetup{
    colorlinks = true,
    citecolor=blue,
    urlcolor=blue,
    breaklinks=true,
    linkcolor = dukeblue,
    linkbordercolor = {white},
}

\usepackage[toc,page]{appendix}
\usepackage{minitoc}

\usepackage{thmtools}
\usepackage{thm-restate}

\usepackage{tocloft}

\usepackage{caption}
\usepackage{multicol}
\usepackage{subcaption}
\usepackage{wrapfig}
\usepackage{mathtools}
\usepackage{arcs}
\usepackage{graphicx}
\usepackage{wasysym}
\usepackage{dsfont,courier}
\usepackage[scaled=.65]{helvet}
\usepackage{tikz}
\usetikzlibrary{positioning,chains,fit,shapes,calc}
\usepackage{bm}
\usepackage{comment}

\definecolor{darkmidnightblue}{rgb}{0.0, 0.2, 0.4}
\definecolor{darkpowderblue}{rgb}{0.0, 0.2, 0.6}
\definecolor{dukeblue}{rgb}{0.0, 0.0, 0.61}
\definecolor{myblue}{RGB}{80,80,160}
\definecolor{mygreen}{RGB}{80,160,80}

\DeclareMathOperator*{\argmin}{arg\,min}

\newcommand*\bigcdot{\mathpalette\bigcdot@{.5}}

\def\hat{\widehat}

\def\z{\zeta}
\def\t{\theta}

\def\k{\kappa}

\def\s{\sigma}

\def\O{\Omega}
\def\ie{\textit{i.e., }}

\def\support{\mathbf{supp}}

\def\prob{\mathbf P}

\def\btheta{\boldsymbol{\theta}}

\def\bX{\boldsymbol{X}}
\def\bXdiese{\boldsymbol{X}^{\text{\tt\#}}}
\def\bfXdiese{{\mathbf X}^{\text{\tt\#}}}

\def\bthetadiese{\boldsymbol{\theta}^\texttt{\#}}

\def\Xdiese{X^{\text{\tt\#}}}

\def\sdiese{\s^{\text{\tt\#}}}

\def\xidiese{\xi^\text{\tt\#}}

\def\tdiese{\t^\text{\tt\#}}

\def\btdiese{\boldsymbol{\theta}^\text{\tt\#}}
\def\tdiese{{\theta}^\text{\tt\#}}
\def\sigmadiese{{\sigma}^\text{\tt\#}}

\def\kall{\bar\k_{\textup{all}}}

\newtheorem{theorem}{Theorem}

\begin{document}

%

\runningtitle{Matching Map Recovery with an Unknown Number of Outliers}

%
\runningauthor{Minasyan, Galstyan, Hunanyan, Dalalyan}

\twocolumn[

\aistatstitle{Matching Map Recovery  
with an Unknown Number of Outliers}

\aistatsauthor{Arshak Minasyan* \And Tigran Galstyan \par \And  Sona Hunanyan \And Arnak Dalalyan }

\aistatsaddress{CREST, ENSAE, IP Paris \\ 5 av. Henry Le Chatelier \\ 91764 Palaiseau 
\\ \href{mailto:arshak.minasyan@ensae.fr}{arshak.minasyan@ensae.fr} 
\And  RAU, YerevaNN\\
  20 Charents street \\ 0025 Yerevan 
  \\ \href{mailto:tigran@yerevann.com}{tigran@yerevann.com}
   \And Yerevan State University\\
  1 Alek Manukyan \\ 0025 Yerevan 
  \\ \href{mailto:hunan.sona@gmail.com}{hunan.sona@gmail.com} 
  \And CREST, ENSAE, IP Paris \\ 5 av. Henry Le Chatelier \\ 91764 Palaiseau 
  \\ \href{mailto:arnak.dalalyan@ensae.fr}{arnak.dalalyan@ensae.fr}
  } 

]

\begin{abstract}
  We consider the problem of finding the matching map between two sets of $d$-dimensional noisy feature-vectors. The distinctive feature of our setting is that we do not assume that all the vectors of the first set have their corresponding vector in the second set. If $n$ and $m$ are the sizes of these two sets, we assume that the matching map that should be recovered is defined on a subset of unknown cardinality $k^*\le \min(n,m)$. We show that, in the high-dimensional setting, if the signal-to-noise ratio is larger than $5(d\log(4nm/\alpha))^{1/4}$, then the true matching map can be recovered with probability $1-\alpha$. Interestingly, this threshold does not depend on $k^*$ and is the same as the one obtained in prior work in the case of $k = \min(n,m)$. The procedure for which the aforementioned property is proved is obtained by a data-driven selection among candidate mappings $\{\hat\pi_k:k\in[\min(n,m)]\}$. Each $\hat\pi_k$ minimizes the sum of squares of distances between two sets of size $k$. The resulting optimization problem can be formulated as a minimum-cost flow problem, and thus solved efficiently. Finally, we report the results of numerical experiments on both synthetic and real-world data that illustrate our theoretical results and provide further insight into the properties of the algorithms studied in this work.
\end{abstract}

\section{INTRODUCTION}\label{sec:intro}
The problem of finding the best matching between two point clouds has been extensively studied, both theoretically and experimentally. The matching problem arises in various applications, for instance in computer vision and natural language processing. In computer vision, finding the correspondence between two sets of local descriptors extracted from two images of the same scene is a well-known example of a matching problem. In natural language processing, in particular, in machine translation, the correspondence between vector representations of the same text in two different languages is another example of a matching problem. Clearly, in these problems, not all the points have their matching point and one can hardly know in advance how many points have their corresponding matching points. The goal of the present work is to focus on this setting and to gain a theoretical understanding of the statistical limitations of the matching problem. 

To formulate the problem and to state the main result, 
let $\mathbf{X} = (X_1, \dots, X_n)$ and $\bfXdiese 
= (\Xdiese_1, \dots, \Xdiese_m)$ be two sequences of feature vectors of sizes $n$ and  $m$ such that 
$m \ge n \ge 2$. We assume that these sequences 
are noisy versions of some feature-vectors, \textit{i.e.},
\begin{equation}\label{model}
\begin{cases}
X_i = \t_i + \s\xi_i \ , \\
\Xdiese_j = \tdiese_j + \sdiese\xidiese_j,
\end{cases}\quad i\in[n] \text{ and } 
j\in[m],
\end{equation}
where $\btheta = (\theta_1, \dots, \theta_n)$ and $\bthetadiese = (\tdiese_1, \dots, \tdiese_m)$ are two sequences of deterministic vectors from $\mathbb{R}^d$, corresponding to the original feature-vectors.  The noise components of $\mathbf X$ and $\bfXdiese$ are two independent sequences of i.i.d.\ isotropic Gaussian random vectors. Formally,
\begin{align}
    \xi_1, \dots, \xi_n, \xidiese_1, \dots, \xidiese_m \stackrel{\text{i.i.d.}}{\sim} \mathcal{N}(0, \mathbf I_d),
\end{align}
where $\mathbf I_d$ is the identity matrix of size 
$d \times d$. We assume that for some $S^* \subset [n]$ of cardinality $k^*$, there exists an injective mapping 
$\pi^* : S^* \to [m]$ such that $\theta_i = 
\tdiese_{\pi^*(i)}$ holds for all $i \in S^*$. We call 
the observations  $(\bX_i:i\in S^*)$ and 
$(\bXdiese_{\pi^*(i)}:i\in S^*)$ \textit{inliers}, while
the other vectors from the sequences $\bX$ and $\bXdiese$  are considered 
to be \textit{outliers}. The ultimate goal is to recover
$\pi^*$ based on the observations $\mathbf X$ and 
$\bfXdiese$ only.

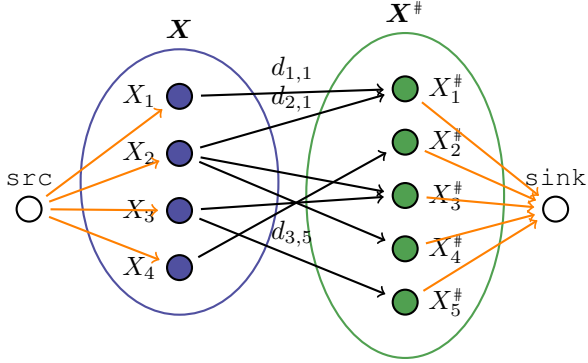
\begin{figure}[t] 
\centering
\raisebox{0.3mm}{
\begin{minipage}[b]{.45\textwidth}
\begin{tikzpicture}[thick,
  fsnode/.style={draw, circle, fill=myblue},
  ssnode/.style={draw, circle, fill=mygreen},
  srcnode/.style={draw, circle},
  every fit/.style={ellipse,draw,inner sep=-2pt,text width=2cm},
  ->,shorten >= 3pt,shorten <= 3pt
  ]

\begin{scope}[start chain=going below,node distance=4mm]
\foreach \i in {1,2,...,4}
  \node[fsnode,on chain] (f\i) [label=left: $X_\i$] {};
\end{scope}

\begin{scope}[xshift=3cm,yshift=0.1cm,start chain=going below,node distance=3.5mm]
\foreach \i in {1, 2, ..., 5}
  \node[ssnode,on chain] (s\i) [label=right: $\Xdiese_{\i}$] {};
\end{scope}

\begin{scope}[xshift=-2cm, yshift=-1.5cm, start chain=going below, node distance=7mm]
\foreach \i in {1}
    \node[srcnode, on chain] (src\i) [label=above:$\texttt{src}$] {};
\end{scope}

\begin{scope}[xshift=5cm, yshift=-1.5cm, start chain=going below, node distance=7mm]
\node[srcnode, on chain] (sink) [label=above:$\texttt{sink}$] {};
\end{scope}

\node [myblue,fit=(f1) (f4),label=above:$\bX$] {};
\node [mygreen,fit=(s1) (s5),label=above:$\bXdiese$] {};

\draw (f1) -- node[above] {$d_{1,1}$} (s1);
\draw (f2) -- (s3);
\draw (f2) -- node[above] {$d_{2,1}$} (s1);
\draw (f2) -- (s4);
\draw (f3) -- (s3);
\draw (f3) -- node[above] {$d_{3,5}$} (s5);
\draw (f4) -- (s2);

\draw[color=orange] (src1) -- (f1);
\draw[color=orange] (src1) -- (f2);
\draw[color=orange] (src1) -- (f3);
\draw[color=orange] (src1) -- (f4);

\draw[color=orange] (s1) -- (sink); 
\draw[color=orange] (s2) -- (sink); 
\draw[color=orange] (s3) -- (sink); 
\draw[color=orange] (s4) -- (sink); 
\draw[color=orange] (s5) -- (sink); 
\end{tikzpicture}
\end{minipage}
}
\caption{Matching as a Minimum 
Cost Flow (MCF) problem. 
The idea is to augment the graph with two nodes,  {\it{source}} and {\it{sink}}, and $n+m$ edges. The capacities of orange edges 
should be set to $1$, while the cost should be set to 
$0$. Setting the total flow sent through the graph 
to $k$, the solution of the MCF becomes a matching 
of size $k$.}
\label{fig:1}
\end{figure}
\raggedbottom
Various versions of this problem have been studied in 
the literature. \cite{jmlr_CD13, collier2016minimax} considered the outlier-free case with equal sizes of
sequences $\bX$ and $\bXdiese$ (\textit{i.e.}, $m = n$ 
and $S^* = [n]$), whereas \cite{galstyan2021optimal}
investigated the case with outliers in one of the
sequences only (\textit{i.e.}, $m \ge n$ and $S^* = [n]$). 
Other variations of the matching problem under 
Hamming loss have been studied by \cite{wang2022random, chen2022one, kunisky2022strong}. These papers obtain 
minimax-optimal separation rates and, in most cases, 
despite the discrete nature of the matching problem, provide computationally tractable procedures to achieve these rates.

When $S^*$ is an arbitrary subset of $[n]$, which is
the setting we focus on in this work, one can wonder whether
the minimax separation rate is the same as in the case 
of known $S^*$. Since the absence of knowledge on $S^*$ 
brings additional combinatorial complexity to the problem, 
one can also wonder whether it is still possible to conciliate statistical optimality and computational tractability. We show in this work that the answers to
these questions are affirmative. 

To explain our result, let us introduce the quantity
\begin{align}
    \kappa_{i,j} = {\| \theta_i - \tdiese_j\|_2}/{ 
    (\sigma^2 +{\sigmadiese}{}^2)^{1/2}},
\end{align}
which is the signal-to-noise ratio of the difference
$X_i-\Xdiese_j$ of a pair of feature-vectors. Clearly, 
for matching pairs this difference vanishes. Furthermore, if $\kappa_{i,j}$ vanishes or is very small
for a non-matching pair, then there is an identifiability issue and consistent recovery of
underlying true matching is impossible. Therefore, a natural condition
for making consistent recovery possible is to assume
that the quantity
\begin{align}
    \kall \triangleq 
    \min_{i\in [n]}\min_{j\in[m]\setminus 
    \{\pi^*(i)\}}
    \kappa_{i,j}
\end{align}
is bounded away from zero. A recovery procedure $\hat\pi$ is considered to be good, if the threshold $\lambda$ such that $\hat\pi$ recovers $\pi^*$ with 
high probability as soon as $\kall\ge \lambda$ is
as small as possible. It was proved in \citep{collier2016minimax} that when $k^*=n=m$, one 
can recover $\pi^*$ with probability $1-\alpha$ for
$\lambda = 4 \big\{\big(d \log (\nicefrac{4n^2}{\alpha})\big)^{1/4} \vee \big(8\log(\nicefrac{4n^2}{\alpha})\big)^{1/2}
\big\}$. Furthermore, it was proved that this threshold
is minimax optimal, \ie optimal in the family of all 
possible recovery procedures. This implies that 
there are two regimes. In the low
dimensional regime $d\lesssim \log n$, the separation rate is dimension independent.  In contrast with this, the separation rate scales roughly as $d^{1/4}$ in the
(moderately) high dimensional regime $d\gtrsim \log n$.

Let us set
\begin{align}\label{kall:rate}
    \lambda_{n,m,d,\alpha} = 
    4 \big\{\big(d \log ({\textstyle\frac{4nm}{\alpha}})\big)^{\nicefrac14}
    \vee \big(8\log({\textstyle\frac{4nm}{\alpha}})\big)^{
    \nicefrac12}\big\}.
\end{align}
The main contributions of this work are the following.
\begin{itemize}
\itemsep0em
    \item For any given $k\in[\min(n,m)]$, 
    we show that the $k$ Least Sum of Squares ($k$-LSS) 
    procedure, based on
    maximizing profile likelihood among matching maps
    between two sets of size $k$, can be efficiently
    computed using the minimum cost flow problem. We denote the matching obtained using $k$-LSS by $\hat\pi^{\textup{LSS}}_k$. 
    
    \item If the value $k$ turns out to be smaller 
    than $k^*$ and $\kall\ge \lambda_{n,m,d,\alpha}$, 
    we prove that $\hat\pi^{\textup{LSS}}_k$ makes no mistake with probability $1-\alpha$. 
    
    \item  We design a data-driven model selection 
    algorithm that adaptively chooses $\hat k$ such
    that with probability $1-\alpha$, we have 
    $\hat{k} = k^*$ and $\hat\pi^{\textup{LSS}}_{
    \hat{k}} = \pi^*$ as soon as $\kall\ge (5/4)\lambda_{n,m,d,\alpha}$.
\end{itemize}
The last item above implies that our data-driven algorithm
$\hat\pi^{\textup{LSS}}_{\hat{k}}$ achieves the 
minimax separation rate. More surprisingly, this
shows that there is no gap in statistical complexities
between the problems of recovering matching maps
in outlier-free and outliers-present-on-both-sides 
settings.

\section{RELATED WORK}
In statistical hypothesis testing, the separation rates became key objects for measuring the quality of statistical procedures, see the seminal papers \citep{Burnashev, Ingster82} as well as the monographs \citep{Ingster,JudNem}. Currently, this approach is widely adopted in machine learning literature \citep{Xing20,Wolfer,Balchard, RamdasISW16,Wei,Collier12a}. Beyond the classical setting of two hypotheses, it can also be applied to multiple testing frameworks, for instance, variable selection \citep{Ndaoud,azais2020multiple,ComDal1,ComDal2} or the matching problem considered here.

In computer vision, feature matching is a well-studied problem. One of the main directions is to accelerate matching algorithms, based on fast approximate methods (see e.g.
\cite{malkov2020,wang2018, harwood2016, jiang2016}). Another direction is to improve the matching quality by considering alternative local descriptors \citep{orb2011, piifd2010, brief2010} for given keypoints. The choice of keypoints is considered in \cite{Tian_2020_ACCV, bai2020d3feat}. 

The minimum cost flow problem was first studied in the context of the Hungarian algorithm \citep{kuhn_hungarian} and the assignment problem, which is a special case of minimum cost flow on bipartite graphs with all edges having unit capacity. Generalization of Hungarian algorithm for graphs with arbitrary edge costs guarantees $\mathcal{O}((n + F) m)$ time complexity, where $n$ is the number of nodes in the graph, $m$ is the number of edges and $F$ is the total flow sent through the graph. There have also been other algorithms with similar complexity guarantees \citep{fulkerson61, ahuja92}. Since then many algorithms have been proposed for solving minimum cost flow problems in strongly polynomial time \citep{OPT93, orlin93, orlinO96, tarjan, galiltardos} with the fastest runtime of around $\mathcal{O}(n m)$. Recent advances for solving MCF problems have been proposed in \cite{goldberg15} and \cite{chen2022maximum}. The latter proposes an algorithm with an almost-linear computational time. 

Permutation estimation and related problems have been recently investigated in different contexts such as statistical seriation \citep{flammarion2019optimal, giraud2021localization,cai2022matrix}, noisy sorting \citep{mao2018minimax}, regression with shuffled data \citep{pananjady2017linear,slawski2019linear}, isotonic regression and matrices \citep{Mao18,pananjady2020isotonic,ma2020optimal}, crowd labeling \citep{ShahBW16a}, recovery of general discrete structure \citep{gao2019iterative}, and multitarget tracking \citep{chertkov2010, kunisky2022strong}.

\section{MAIN THEORETICAL RESULT}
\label{sec:theoreticalResults}

This section contains the main theoretical contribution
of the present work. In order to be able to recover
$S^*$ and the matching map $\pi^*$, the key ingredient
we use is the maximization of the profile likelihood.
This corresponds to looking for the least sum of 
squares (LSS) of errors over all injective mappings
defined on a subset of $[n]$ of size $k$. Formally, 
if we define 
\begin{align}
    \mathcal{P}_k :=  \bigg\{ \pi:S \to [m] \, 
    \text{such that } \ 
    \begin{matrix}
    		S\subset [n], |S| = k, \\ 
    		\pi \text{ is injective}
	\end{matrix}    
    \bigg\}
\end{align}
to be the set of all $k$-matching maps, we can 
define the procedure $k$-LSS as a solution to the optimization problem
\begin{align}\label{eq:lss-k}
    \widehat{\pi}_{k}^{\textup{LSS}} \in \argmin_{\pi \in \mathcal{P}_k} \sum_{i \in S_\pi} \| X_i - \Xdiese_{\pi(i)} \|^2_2,
\end{align}
where $S_\pi$ denotes the support of function $\pi$. 
In the particular case of $k^*=n$, the optimization above
is conducted over all the injective mappings from $[n]$ to $[m]$. This coincides with the LSS method from \citep{galstyan2021optimal}.

Let  $\hat \Phi(k)$ be the error of $\widehat{\pi}_{k}^{\textup{LSS}}$, that is
\begin{align}\label{hat-Phik}
    \hat \Phi(k) = 
     \min_{\pi \in \mathcal{P}_k} 
     \sum_{i \in S_{\pi}} \nolimits
     {\| X_i - \Xdiese_{\pi(i)} \|^2_2}.
\end{align}
For some values of tuning parameters $\lambda>0$ and 
$\gamma > 0$, as well as for some  $k_{\min}\in [n]$, 
initialize $k\leftarrow k_{\min}$ and
\begin{enumerate}
\itemsep0em
    \item Compute $\hat\Phi(k)$ and $\hat\Phi(k+1)$.
    \item Set $\bar{\sigma}_{k}^2 = 
     \hat{\Phi}(k)/(kd)$. 
    \item If $k=n$ or $\hat{\Phi}(k+1) - \hat{\Phi}(k) 
    > \frac{d + \lambda}{1-\gamma} \bar{\sigma}_{k}^2$,
    \item[] then output $(k,\bar\sigma_k,\hat{\pi}_k^{\textup{LSS}})$. 
    \item Otherwise, increase $k \leftarrow k + 1$ and go to 
    Step 1. 
\end{enumerate}\label{joint_est_algo}
In the sequel, we denote by $(\hat k,\bar\sigma_{\hat k}, \hat{\pi}_{\hat k}^{\textup{LSS}})$ the output 
of this procedure. Notice that we start with the value of $k=k_{\min}$, which in the absence of any information on the number of inliers might be set to $k=1$. 
However, using a higher value of $k_{\min}$ might considerably speed up the procedure and improve its quality.

For appropriately chosen values of $\gamma$ and $\lambda$, as stated in the next theorem, the 
described procedure outputs the correct values 
of $k^*$ and $\pi^*$ with high probability.

\begin{theorem}\label{thm:main}
Let $\alpha \in (0, 1)$  and 
$\lambda_{n,m,d,\alpha}$ be defined by \eqref{kall:rate}. 
If $\kall > (\nicefrac{5}{4}) \, \lambda_{n,m, d,
\alpha}$, then the output $(\hat{k}, \hat{\pi}_{\hat{k}}^{ 
\textup{LSS}})$ of the model selection algorithm with 
parameters $\lambda = (\nicefrac14) \lambda_{ 
n,m,d,\alpha}^2, \gamma = \nicefrac{\lambda}{d}$ satisfies $\prob(\hat{\pi}_{\hat{k}}^{ 
\textup{LSS}} = \pi^*) 
\ge 1 - \alpha$. 
\end{theorem}

Since the condition on the separation distance $\kall$ compared to the case of known $k^*$ is different by only a slightly larger constant, from the perspective of statistical accuracy, the case of unknown $k^*$ 
is not more challenging than that of the known $k^*$. 

In the sequel, without much loss of generality, we 
assume that the sizes of $\bX$ and $\bXdiese$ are equal, \ie{$n=m$}. Indeed, in the case, $m> n$ one can add $m-n$ points arbitrarily far from the rest of the points to the smaller set $\bX$ obtaining equal size sets ${\bX}^+$ and $\bXdiese$.
 
Notice that in the optimization problem \eqref{eq:lss-k} the domain of $\pi$ is a finite set of injective functions. For a given value of $k$, the number of such functions is $ k! \binom{n}{k}^2$ making thus an exhaustive search computationally infeasible. Instead, we show in \Cref{Sec:num} that the optimization problem formulated in \eqref{eq:lss-k} can indeed be solved efficiently with complexity $\widetilde{\mathcal{O}}(\sqrt{k}\,n^2)$, where the notation $\widetilde{\mathcal{O}}$ hides polylogarithmic factors, \ie up to polylogarithmic factors, the computational cost is of order $\sqrt{k} n^2$.

\section{INTERMEDIATE RESULTS AND PROOF OF THEOREM 
\ref{thm:main}}
\label{sec:main_proof}
This section is devoted to the proof of our main result. Along the way, we establish some intermediate results
which are of interest on their own. The proofs of
some technical lemmas are deferred to \Cref{app-A}.

\subsection{Sub-mapping Recovery by LSS for \texorpdfstring{$k \le k^*$}{}}

The first question we address in this section is under 
which conditions the LSS estimator $\widehat{\pi}_{k}^{\textup{LSS}}$ from \eqref{eq:lss-k} recovers correct matches. Of course, the only way of correctly estimating the true matching is to choose $k = k^*$. However, it turns out that even if we overestimate the number of outliers and choose a value $k$ which is smaller than the true value $k^*$, with high probability the $k$-LSS estimator makes no wrong matches. Naturally, this result, stated in the next theorem, is valid under the condition that the relative signal-to-noise ratio of all incorrect pairs of original features is larger than some threshold. 

\begin{theorem}[Quality of $k$-LSS when $k\le k^*$]
\label{theorem_2}
    Let $\widehat{S} = \support(\widehat\pi)$
    for $\widehat\pi = {\widehat\pi}_{k}^{\textup{LSS}}$
    defined by \eqref{eq:lss-k}, 
    $\alpha \in (0,1)$ and 
    \begin{align}\label{lambda}
        \lambda_{n,d,\alpha} = 4 \Big(\big(d \log (\nicefrac{4n^2}{\alpha})\big)^{1/4} \vee \big(8\log(\nicefrac{4n^2}{\alpha})\big)^{1/2}
    \Big).
    \end{align}
    If $k\le k^*$ and the signal-to-noise 
    ratio satisfies the condition $\kall \ge 
    \lambda_{n,d,\alpha}$ then, with probability at least $1-\alpha$, 
    the support of the estimator $\widehat{\pi}$ is included in $S^*$ and $\widehat{\pi}$ coincides with $\pi^*$ on the set $\widehat{S}$. Formally,
    \begin{align}
        \mathbf{P}\big( 
        \widehat{S} \subset S^* \text{ and }  \widehat\pi(i) = \pi^*(i), 
        \forall i \in \widehat{S} \, \big) \ge 1- \alpha.
    \end{align}
\end{theorem}

\begin{proof}[Proof of \Cref{theorem_2}]
Note that the random vectors 
\begin{align}
    \eta_{ij} = 
    {(\sigma\xi_i - \sigmadiese\xidiese_j)}/{\sqrt{\sigma^2+{\sigmadiese}{}^2}}
\end{align}
are standard Gaussian and define the following quantities 
\begin{align}
\label{def:zeta1zeta2}
\begin{gathered}
    \zeta_1 \triangleq \max_{i, j \neq \pi^*(i)} \frac{|(\theta_i - \tdiese_j)^\top\eta_{ij}|}{\| \theta_i - \tdiese_j\|_2}, \\ 
    \zeta_2 \triangleq d^{-1/2} \max_{i, j} \big|\|\eta_{ij}\|^2_2 - d \big|.
\end{gathered}
\end{align}
For the ease of notation, for any matching map $\pi$ we also define $L(\pi)$ as follows
\begin{align}
    L(\pi) = \sum_{i \in S_{\pi}}\nolimits \frac{\| 
    X_i - \Xdiese_{\pi(i)}\|_2^2}{\sigma^2 + \sdiese{}^2}.
\end{align}
We start with two auxiliary lemmas that 
will be used in other proofs as well. The proofs of 
these lemmas are deferred to the appendix.

\begin{restatable}{lemma}{lemmauno}
\label{lem:1}
Let $\pi$ be any matching map that can not be
obtained as a restriction of $\pi^*$ on a subset 
of $[n]$. Let $S_0\subset S^*$ be an arbitrary 
set satisfying 
$|S_0|\le |S_\pi|$ and $\{i\in S_\pi\cap S^*:\pi(i) = \pi^*(i)\}
\subset S_0$ and let $\pi_0$ be the restriction 
of $\pi^*$ to $S_0$. On the event $\Omega_0 = \{8\zeta_1 
\le \kall ; 4\sqrt{d}\,\zeta_2 \le \kall^2\}$, 
we have 
\begin{align}
    L(\pi) - L (\pi_0) & \ge (\nicefrac14)\kall^2 +
    d(|S_\pi| - |S_0|). 
\end{align}
\end{restatable}

Let $\pi$ be any matching map from $\mathcal P_{k}$ that is not a restriction of $\pi^*$. Since $|S_\pi|
= k \le k^*$, there exists necessarily a $\pi_0$ 
as in \Cref{lem:1} such that $|S_0| = |S_\pi|$. 
For this $\pi_0$, we have 
$L(\pi) - L (\pi_0)  \ge (\nicefrac14)\kall^2 >0$. 
This implies that $\pi$ cannot be a minimizer
of $L(\cdot)$ over $\mathcal P_k$. As a consequence, on $\Omega_0$, any minimizer of  $L(\cdot)$ over
$\mathcal P_k$ is a restriction of $\pi^*$. 
Therefore, on $\O_0$, we have $\hat S \subset S^*$ 
and $\hat\pi_k = \pi^*|_{\hat S}$. It remains to
prove that $\prob(\Omega_0)\ge 1-\alpha$.

\begin{restatable}{lemma}{lemmados}
\label{lem:2}
Let $\Omega_{0, x} = \{8\z_1 \le x\} \cap \{4\sqrt{d}\z_2 \le x^2\}$ with $\z_1, \z_2$ defined as in \eqref{def:zeta1zeta2}. Then, for every $x > 0$, 
$\prob(\O_{0, x}^\complement)$ is upper bounded by
\begin{align}
     2n^2 \Big( \exp\Big\{-\frac{x^2}{128}\Big\} + 
    \exp\Big\{ - \frac{x^2}{128d}\big(x^2 \wedge 4d\big) \Big\}\Big).
\end{align}
\end{restatable}
We apply \Cref{lem:2} with $x = \kall$ 
to show that $\prob(\O_0) \ge 1-\alpha$.
Clearly, a sufficient condition for the latter is 
\begin{align}
    \begin{cases}
     2n^2 \exp \left\{ -\kall^2/128 \right\} \leq \alpha/2, \\
     2n^2 \exp\bigg\{ - \frac{(\kall/16)^2}{d}\big(2\kall^2 \wedge 8d\big) \bigg\} \leq \alpha/2.
    \end{cases}
\end{align}
This system is equivalent to
\begin{align}
    \kall \geq 8 \Big(2\log \frac{4n^2}{\alpha}\Big)^{1/2} \quad \text{and} \quad \kall 
    \geq 4 \Big( \frac{d}2 
    \log \frac{4n^2}{\alpha}\Big)^{1/4}.
\end{align}
Therefore, if the signal-to-noise ratio satisfies 
\begin{align}
    \kall \ge 4 \Big(\big(d \log (\nicefrac{4n^2}{\alpha})\big)^{1/4} \vee \big(8\log(\nicefrac{4n^2}{\alpha})\big)^{1/2}
    \Big),
    \label{eq:kappa_condition}
\end{align}
we have $\prob(\O_0) \ge 1-\alpha$.
\end{proof}

\subsection{Matching Map Recovery for Unknown \texorpdfstring{$k^*$}{}} 
\label{sec:unknown-k}

If no information on $k^*$ is available, and the goal is to
recover the entire mapping $\pi^*$, one can proceed by 
model selection. More precisely, one can compute the 
collection of estimators $\{\hat\pi_k^{\textup{LSS}}: k\in[n]\}$ and select one of those using a suitable 
criterion. To define the selection criterion proposed 
in this paper, let us remark that 
\begin{align}
    \hat \Phi(k) = 
     \min_{\pi \in \mathcal{P}_k} 
     \sum_{i \in S_{\pi}} \nolimits
     {\| X_i - \Xdiese_{\pi(i)} \|^2_2}
\end{align}
is an increasing function. 
The increments of this function for $k\le k^*$ 
are not large, since they essentially 
correspond to the squared norm of a pure noise vector 
distributed according to a scaled $\chi^2$ distribution 
with $d$ degrees of freedom. The main idea behind the criterion we propose below is that the increment of 
$\hat \Phi$ at $k ^*$ is significantly larger than 
the previous ones and the gap is of order $\kall^2$. Therefore, if $\kall^2$ is larger than
the deviations of the $\chi^2_d$ distribution, we
are able to detect the value of $k^*$ and to 
estimate the true matching.

Based on these considerations, for any tolerance level $\alpha\in (0,1)$, we set $\sigma_0^2 = \s^2 + \sdiese{}^2$ and define the estimator\footnote{We 
use the  convention $\hat \Phi(0) = 0.$}
\begin{align}
    \hat k = 1 + \max\big\{k\in\{0,&\ldots,n-1\} 
    : \hat \Phi(k+1) -\hat \Phi(k) \\
   & \le \sigma_0^2 
    (d + \lambda_{n,d,\alpha}^2/4) \big\}
\end{align}
with $\lambda_{n,d,\alpha}$ as in 
\eqref{lambda}.

\begin{theorem}[Model selection accuracy]\label{thm:3}
Let $\alpha\in(0,1)$. If  
$\kall > \lambda_{n,d,\alpha}$, then it holds that $\prob\big(\hat k = 
k^*\text{ and }\hat\pi_{\hat k} = \pi^*\big ) 
\ge  1-\alpha$. Therefore, $\lambda_{n,d,\alpha}$ is an upper bound on the separation
distance in the case of unknown $k^*$. 
\end{theorem}

A remarkable feature put forward by this result is that
a data-driven selection of $k$ based on 
the increments of the test statistics $\hat\Phi$ leads 
to the recovery of $\pi^*$, with high probability, under
the same constraint on the separation rate as in the case of known $k^*$. It is however important to
underline 
that this criterion requires the knowledge of the noise
level. Therefore, from the point of view of 
statistical accuracy, the case of
unknown $k^*$ is not more difficult than the case of 
known $k^*$, provided the noise levels are known. It is also worth mentioning that our procedure requires only $\sigma_0^2 = \s^2 + \sdiese{}^2$, not $\sigma$ and $\sdiese$ separately. 

\begin{proof}[Proof of \Cref{thm:3}]
The main parts of the proof will be done in the following two lemmas, the proofs of which are 
postponed to the appendix. For the known value of $\sigma_0^2$ it is more convenient to work with the normalized version of test statistics $\hat\Phi(\cdot)$, 
denoted by $\hat L(\cdot)$ and defined by
\begin{align}
    \hat L(k) = \min_{\pi \in \mathcal{P}_k}
     \sum_{i \in S_{\pi}} \nolimits
     \frac{\| X_i - \Xdiese_{\pi(i)} \|^2_2}
     {\sigma^2 + \sigmadiese{}^2} \equiv \frac{\hat \Phi(k)}{\sigma_0^2}.
\end{align}

\begin{restatable}{lemma}{lemmatres}
\label{lem:5}
On the event, $\O_0 = \{8\zeta_1 \le \kall;4\sqrt{d}\,\zeta_2\le \kall^2\}$, we have $\hat L
(k^*+1) - \hat L(k^*) \ge d + \kall^2/4$. 
\end{restatable}

\begin{restatable}{lemma}{lemmaquatro}
\label{lem:6}
On the event, $\O_0 = \{8\zeta_1 \le \kall;4\sqrt{d}\,\zeta_2\le \kall^2\}$, for every 
$k< k^*$, we have 
$\hat L(k+1) - \hat L(k) \le  d + \sqrt{d}\,\zeta_2$. 
\end{restatable}

\Cref{lem:2} implies that the probability of the event
$\Omega_1 = \{8\zeta_1 \le \lambda_{n,d,\alpha};
4\sqrt{d}\,\zeta_2 \le \lambda^2_{n,d,\alpha}\}$ 
is at least $1-\alpha$. Since 
$\Omega_1$ is included in $\O_0$, in view of 
\Cref{lem:6}, on $\Omega_1$ we have $\hat L(k+1) 
- \hat L(k) \le d + \lambda^2_{n,d,\alpha}/4$ for 
any $k<k^*$. On the 
other hand, in view of \Cref{lem:5}, on the same 
event we have $\hat L(k^*+1)-\hat L(k^*) \ge d 
+ \kall^2/4 > d + \lambda^2_{n,d,\alpha}/4$. 
This implies that $\hat k = k^*$ and, therefore, 
$\hat\pi_{\hat k} = \hat\pi_{k^*}$. Thanks to 
\Cref{theorem_2}, on the same event $\O_1$, we have 
$\hat\pi_{k^*} = \pi^*$.
\end{proof}

\subsection{Matching Map Recovery for Unknown \texorpdfstring{$k^*$}{} and Unknown Noise Level}\label{ssec:joint}

In the previous subsection, we considered the case 
of unknown $k^*$ with known noise levels $\sigma$ 
and $\sdiese$. Notice that we do not need to estimate 
parameters $\sigma, \sdiese$ separately, it is 
sufficient to estimate only their squared sum, 
which is denoted by $\sigma_0^2$. In the definition 
of $\hat{k}$, we use the value of $\sigma_0^2$ 
in the threshold for $\hat \Phi(k+1) -\hat \Phi(k)$. 
When both $k^*$ and $\sigma_0^2$ are unknown, 
we first estimate $\sigma_0^2$ and then plug it 
in the selection criterion of $k$. 

Thus, we define ``candidate'' estimators of $\sigma_0^2$ 
\begin{align}\label{eq:sigma_est}
    \big\{\bar{\sigma}_k^2 = \frac{\hat{\Phi}(k)}{kd},\ 
    k\in [n]\big\}.
\end{align}
The rationale for this definition is that for
small values of $k$, $\hat\pi_k^{\textup{LSS}}$ 
contains only correct matches and, therefore, 
$\hat\Phi(k)/\sigma_0^2$ is merely a sum of $k$ 
independent random variables drawn from the 
$\chi^2$ distribution with $d$ degrees 
of freedom. Hence, after division by ${kd}$, 
we obtain an estimator of $\sigma_0^2$. However, 
from the perspective of testing the values of $k$, 
we need to slightly overestimate the noise 
variance. This is done through the multiplication
by the inflation factor $1/(1-\gamma)$. 

We are now ready to proceed with the proof of our main result stated in \Cref{thm:main}.

\begin{proof}[Proof of \Cref{thm:main}]
We will provide the proof only in the high dimensional setting, that is we assume throughout the proof that $d \ge 800\log(2n/\sqrt{\alpha})$ . First, we show that for every $k < k^*$ the condition from $\hat{\Phi}(k+1) - \hat{\Phi}(k) \le \frac{d + \lambda}{1-\gamma} \bar{\sigma}_{k}^2$ is satisfied on an event of high probability. Second, we prove that for $k=k^*$ this condition is violated on the same event of high probability. Therefore, the combination of these two results concludes the proof. 

Using the first part of the proof of \Cref{lem:6}, for all $k <k^*$ on the event $\Omega_0 = \{8\zeta_1 
\le \lambda ; 4\sqrt{d}\zeta_2 \le \lambda^2\}$, 
we have 
\begin{align}
    \frac{\hat \Phi(k+1) - \hat \Phi(k)}{\hat \Phi(k)} 
    &= \frac{\sum_{\hat S_{k+1}} \|\eta_{i,\pi^*(i)}\|^2_2 - 
    \sum_{\hat S_{k}} \|\eta_{i,\pi^*(i)}\|^2_2}{
    \sum_{i\in\hat S_{k}} \|\eta_{i,\pi^*(i)}\|^2_2}\\
    &\le \frac{d + \sqrt{d}\zeta_2}{kd + k\sqrt{d} 
    \cdot \min_{1\le i \le n} \frac{\|\eta_{i,\pi^*(i)}
    \|^2_2-d}{\sqrt{d}}}\\
    &\le \frac{d + \sqrt{d}\zeta_2}{k(d-\sqrt{d} \zeta_2)_+}. 
\end{align}
Using the second part of the proof of \Cref{lem:2}, 
we can further upper bound the expression from the 
last display as follows 
\begin{align}
     \frac{\hat \Phi(k+1) - \hat \Phi(k)}{\hat \Phi(k)} 
     \le \frac{d + \lambda^2/4}{k(d - \lambda^2/4)_+}. 
\end{align}
Now we show that for $k=k^*$ the relative difference of 
function $\hat \Phi(\cdot)$ at points $k^*+1$ and $k^*$ 
is large enough. Indeed, we have 
\begin{align}
\frac{\hat \Phi(k^*+1) - \hat \Phi(k^*)}{\hat 
    \Phi(k^*)} &\ge 
    \frac{\hat \Phi(k^*+1) - \hat \Phi(k^*)}{\sum_{i\in S^*} \|X_i - \Xdiese_{\pi^*(i)}\|^2_2}\\
    &\ \ge \frac{\min_{i \neq \pi^*(j)} \|X_i - \Xdiese_j\|_2^2}{\sigma_0^2\sum_{i\in S^*} \|\eta_{i, \pi^*(i)}\|_2^2}\\ 
    &\ \ge \frac{\kall^2-2\zeta_1\kall+d-\sqrt{d}\zeta_2}{k(d+\sqrt{d}\zeta_2)},
\end{align}
where the first inequality follows from the definition of function $\hat{\Phi}(\cdot)$, while second and third inequalities are consequences of the definitions introduced earlier in this section.
Then, on the event $\Omega_0$ we bound the quantities $\zeta_1$ and $\zeta_2$ from the last display along with using the condition on $\kall$. One can now check that if $d \ge 800\log(2n/\sqrt{\alpha})$ then $\lambda^2 \le \nicefrac{4}{5} \, d$, which in turn implies 
\begin{align}
    \frac{\hat \Phi(k^*+1) - \hat \Phi(k^*)}{\hat \Phi(k^*)} &
    \ge \frac{\kall^2- \kall \lambda / 4 + d - \lambda^2/4}{k(d + \lambda^2/4)}\\ 
    &\ge \frac{d + \lambda^2/4}{kd(1-\lambda^2/4d)}.
\end{align}
Thus, we have shown that on the event $\Omega_0$ our model selection procedure will select $k^*$, \ie $\hat{k} = k^*$. The last equality implies that $\hat{\pi}_{\hat{k}} = \hat{\pi}_{k^*}$. Moreover, in view of \Cref{theorem_2} on the same event $\Omega_0$ we have $\hat{\pi}_{k^*} = \pi^*$. Finally, using \Cref{lem:2}, we get that the event $\Omega_0$ occurs with probability at least $1-\alpha$. Therefore, the desired result follows.
\end{proof}

\subsection{Lower Bounds}

\Cref{theorem_2} and 
\Cref{thm:3} imply that the minimax rate of separation
in the problem of recovering $\pi^*$ is at most of the
order of $\lambda_{n,m,d,\alpha}$ defined in \eqref{kall:rate}. 
An interesting and natural question is whether this rate is optimal. In the literature, the lower bounds for similar models have been proved, see \cite[Theorem 2]{collier2016minimax} for the case of $n=m$ and \cite[Theorem 5]{galstyan2021optimal} for the general rectangular case $m \ge n$. Our model is more general\footnote{Indeed, it involves an additional (unknown) parameter $k^*$.} than those of  
these two references, the same lower bound applies to 
our model. Therefore, combining the results of \Cref{thm:3} and \citep[Theorem 2]{collier2016minimax} along with the fact that the separation distance has the same rate in both theorems implies that $\lambda_{n,m,d,\alpha}$ is the optimal separation rate. 


\section{COMPUTATIONAL ASPECTS AND NUMERICAL EXPERIMENTS}\label{Sec:num}

In this section, we address computational aspects of the optimization problem from \eqref{eq:lss-k}. We show that it can be cast into a minimum cost flow problem. The latter is also known as an imperfect matching problem and, to the best of our knowledge, the fastest algorithm with complexity ${O}(\sqrt{k^*}\,n^2\log(k^*))$ is proposed in \citep{goldberg15}. We then report the results of numerical experiments conducted on both synthetic and real data and highlight their relation to the aforementioned theorems stated and proved in Sections \ref{sec:theoreticalResults} and \ref{sec:main_proof}, respectively. Our reproducible codes are provided in the supplementary material.

\subsection{Relation to Minimum Cost Flow Problem}
Let $d_{ij} = \| X_i - \Xdiese_j \|^2_2$, for $(i, j) \in [n] \times [m]$, be the squared distances between observed feature-vectors. Consider the following linear program
\begin{align}\label{lss-simplex1}
    \text{minimize} \,\, &\sum_{i=1}^n \sum_{j=1}^m d_{ij} w_{ij}
\end{align}
subject to $\boldsymbol w = (w_{ij})_{(i, j) \in [n] \times [m]} \in [0, 1]^{n\times m}$ satisfying 
\begin{align}\label{lss-simplex2}
    \sum_{i=1}^n w_{i\boldsymbol{\cdot}} &\le 1, \quad 
    \sum_{j=1}^m w_{\boldsymbol{\cdot} j} \le 1, \quad
    \sum_{i=1}^n \sum_{j=1}^m w_{ij} = k,
\end{align}
known as the minimum cost flow problem. 
Above, the notation $\sum_{i=1}^n w_{i\boldsymbol{\cdot}} \le 1$ means that 
$\sum_{i=1}^n w_{ij} \le 1$ for all $j \in [m]$, and similar convention is used for $\sum_{j=1}^m w_{\boldsymbol{\cdot}j} \le 1$. 

The formulation as an MCF problem is obtained by adding two auxiliary nodes to the graph, called {\it{source}} and {\it{sink}} (see Fig.~\ref{fig:1}). We are interested in the flow of the minimal cost, where the cost of each edge except those adjacent to {\it{source}} and {\it{sink}} is assigned from the distance matrix $\{d_{ij}\}_{i,j=1}^{n,m}$. The cost of the rest of the edges is equal to $0$. The capacity that can be sent through each edge is equal to $1$. The supply of {\it{source}} and {\it{sink}} are $k$ and $-k$, respectively. The solution of  
\eqref{lss-simplex1} given the constraints \eqref{lss-simplex2} provides the weights $\{w_{ij}\}_{i,j=1}^{n,m}$, from which the matching $\widehat{\pi}_{k}^{\textup{LSS}}$ can be recovered. Indeed, if $w_{ij}= 1$ then $X_i$ and $\Xdiese_j$ are matched. The last constraint in \eqref{lss-simplex2} implies that the matching
size (number of $w_{ij}$ that are equal to $1$) will 
be $k$. Though the algorithm provided in \citep{goldberg15} has the fastest known asymptotic complexity, the implementation of their algorithm is out of the scope of this paper. Therefore, in our experiments, we used \texttt{SimpleMinCostFlow} solver from OR-tools library \citep{ortools}.

\subsection{Numerical Experiments on Synthetic Data}
\label{ssec:syntheticData}

In this part, we conducted several experiments on synthetic data to support our theoretical findings. In these experiments, we constructed two sets of sizes $n=m=100$ consisting of $d=100$-dimensional data points. The underlying matching size $k^*$ was set to $60$. In other words, in each point cloud, we had $60$ \textit{inliers} and $40$ \textit{outliers}. We also fixed the confidence level at $1\%$, \ie $\alpha=0.01$. The procedure for generating synthetic data was as follows. We set $S^* = [k^*]$ and chose an additional parameter $\tau$ used to control $\kall$ throughout 
the experiments. Then, each coordinate of $\btheta$ and $\btdiese$ was independently sampled from a Gaussian distribution 
with $0$ mean and standard deviation $\tau$. Additionally, for every $i \notin S^* $, 
we incremented every coordinate of $\theta_i$ by 
$\tau$, \ie $\theta_{i} \leftarrow \theta_{i} + 
\tau\bm{1}$ and, for every $j \notin \text{Im}(\pi^*)$,
we incremented every coordinate of $\tdiese_j$ by $2\tau$. This allows the outliers of each point cloud to be sufficiently far from each other, hence a pair of outliers is less likely to be confused as a pair of inliers. Notice also that such a choice of generating outliers fits within the conditions of \Cref{thm:main}. Finally, the sequences $\bX$ and $\bXdiese$ were generated according to \eqref{model} with $\pi^*(i) = i$, for all $i\in S^*$. We measured the performance of our estimator $\widehat{\pi}_{k^*}^{\textup{LSS}}$ by its precision, which is the number of correctly matched inliers divided by $k = k^*$. 
\begin{figure}[ht]
    \centering
    \includegraphics[width=0.49\textwidth, height=2.8cm]{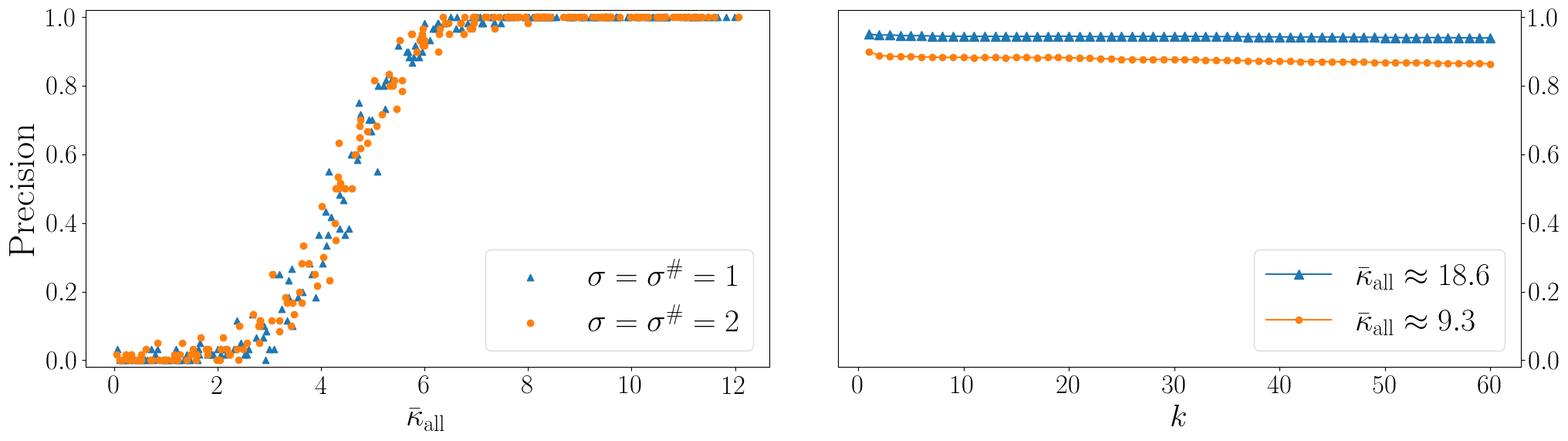}
    \caption{
    Left: the dependence of the matching precision of $\widehat{\pi}_{k^*}^{\textup{LSS}}$ on $\kall$ for  $\sigma=\sdiese=1$ (blue triangles) and $\sigma=\sdiese=2$ (orange circles). Right: the accuracy of subset recovery for a known $k \le k^*$. In both plots, for $\kall$ large enough we observe that the estimated matching $\widehat{\pi}_{k^*}^{\textup{LSS}}$ indeed coincides with $\pi^*$ or yields a subset of it (right plot).}
    \label{fig:theorem_1_prec}
\end{figure}

Let us now comment on the results obtained in Fig. \ref{fig:theorem_1_prec}. On the left plot, given that $\kall$ is large enough we observe that $\widehat{\pi}_{k^*}^{\textup{LSS}} = \pi^*$ holds. Another interesting feature that can be inferred from Fig. \ref{fig:theorem_1_prec} is that the precision is independent of the noise levels $\s, \sdiese$. This feature is consistent with \Cref{thm:main} through the definition of $\lambda_{n, m, d, \alpha}$ which is independent of $\s, \sdiese$. 
The right plot upholds the findings of \Cref{theorem_2}. Indeed, for some $k \le k^*$ the support of $\widehat{\pi}_{k}^{\textup{LSS}}$ will be included in the support of $\pi^*$, plus, the values of $\widehat{\pi}_{k}^{\textup{LSS}}$ and $\pi^*$ coincide on this support. Moreover, by plugging the values of this experiment into \eqref{lambda}, we get $\lambda_{n, d, \alpha} \approx 44$, which means that the results proved in Theorems \ref{thm:main} and \ref{theorem_2} hold whenever $\kall \ge \nicefrac{5}{4}\lambda_{n, d, \alpha}$ and $\kall \ge \lambda_{n,d, \alpha}$, respectively. The results shown in Fig. \ref{fig:theorem_1_prec} are the average values over $200$ independent trials of the same experiment. 




In the previous experiment, we focused on the case of known $k^*$ and noise levels. First, we consider the case of known noise levels and apply the estimator proposed in \Cref{sec:unknown-k}. In Fig. \ref{fig:k_alpha}, we illustrate how the unknown value of $k^*$ is estimated for two different known noise levels. Namely, given that $\kall$ is ``small'' ($\kall < 15$) all the pairs are considered as inliers and $\widehat{k} = 100$, as opposed to the experimental setup, where $k^* = 60$. However, given that the value of $\kall$ is large enough, the algorithm starts to differentiate inliers from outliers, hence we have $\widehat{k} = 60$ starting from $\kall \approx 22$, which confirms the result of \Cref{thm:3}. Recall that the result of \Cref{thm:3} holds whenever $\kall \ge \lambda_{n, d, \alpha} \approx 44$. Moreover, similar to the left plot of Fig. \ref{fig:theorem_1_prec}, the threshold after which the estimation becomes exact does not depend on noise levels. The results reported in Fig. \ref{fig:k_alpha} are the average values over $200$ independent trials of the same experiment. 

Finally, we consider the case when no additional information is available neither about $k^*$ nor about $\s^2_0$, which corresponds to the most generic setting of \Cref{thm:main}. Recall the sequential procedure from \Cref{sec:theoreticalResults} of estimating the triplet $(k^*,\sigma_0^2,\pi^*)$. In Fig. \ref{fig:k_alpha_sigma_joint}, we plot the dependencies of estimates of $k^*$ and $\s_0^2$ on $\kall$. There are several key features that are worth noticing. First, for small values of $\kall$ it is impossible to distinguish inliers from outliers, hence all points are treated as inliers and $\widehat{k} = 100$. As a consequence $\s_0^2$ is overestimated. Second, for large enough values of $\kall$ both $k^*$ and $\s_0^2$ are accurately estimated, therefore the precision of $\widehat{\pi}_{\widehat{k}}^{\textup{LSS}}$ is (close to) $1$. To make a link with \Cref{thm:main} we also include the theoretical value of threshold $\lambda_{n, m, d, \alpha}$ in the plots. 
\begin{figure}
    \centerline{\includegraphics[width=0.49\textwidth]{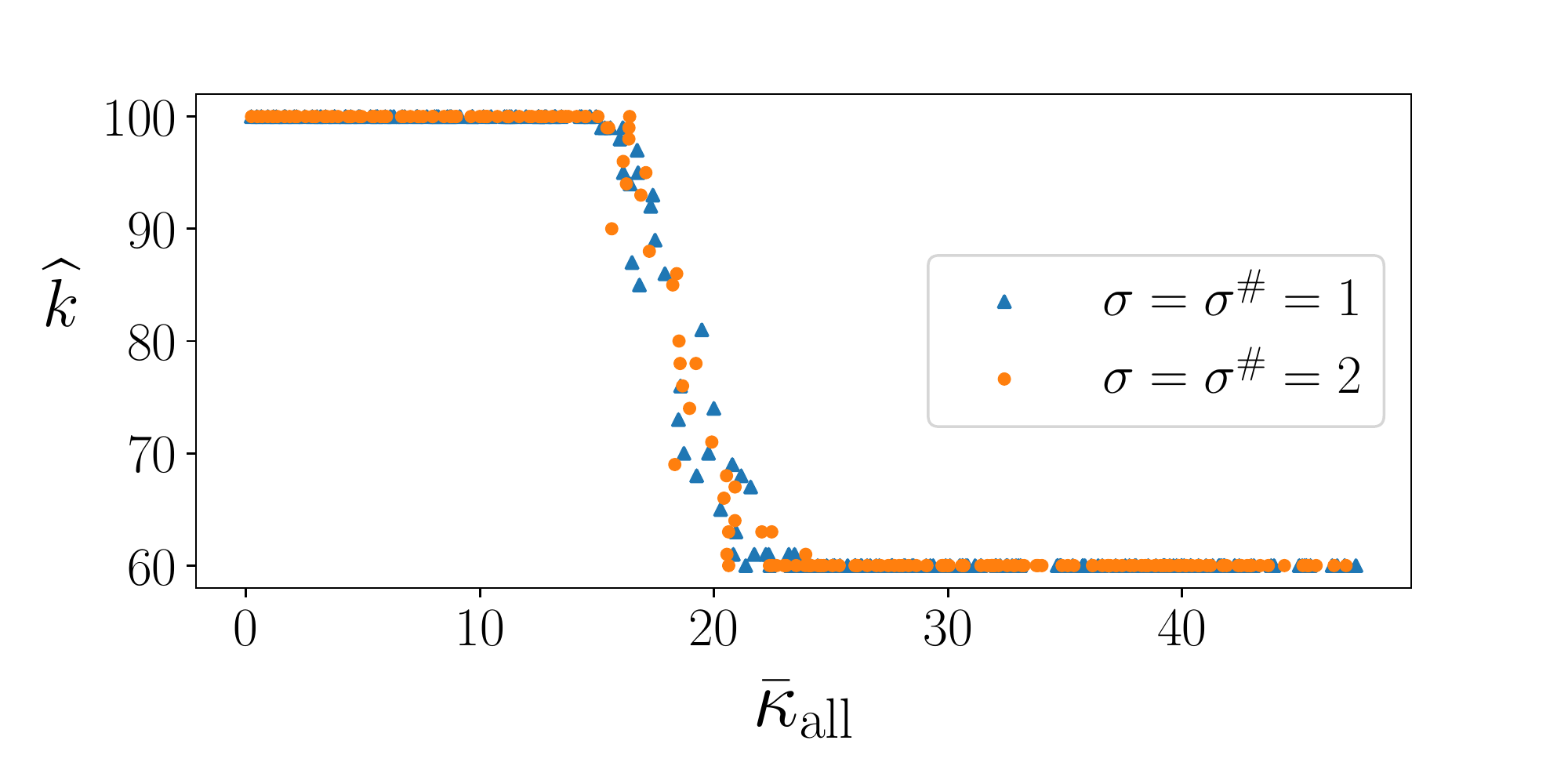}}
    \vspace{-12pt}
    \captionof{figure}{
    Dependence of the estimate of $k^*$ on $\kall$ when the noise levels $\s$ and $\sdiese$ are known. Observe that given that the value of $\kall$ is large enough our procedure recovers the true value of $k^*=60$ while failing to identify the outliers when the signal-to-noise ratio is small and hence estimating $\widehat{k} = 100$. 
    }
    \label{fig:k_alpha}
    \vspace{-4pt}  
\end{figure}  

\begin{figure}[ht]
    \centering
    \includegraphics[width=0.47\textwidth, scale=.5]{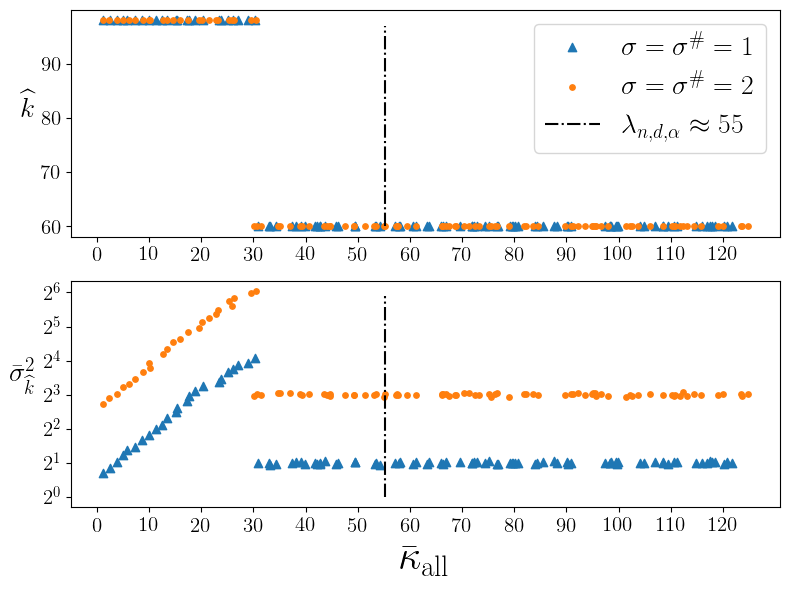}
    \vspace{-12pt}
    \caption{Two settings were considered: $\sigma,\sdiese=1$ (blue triangles) and 
    $\sigma,\sdiese=2$ (orange circles). The top plot shows the dependence of the estimate of $k^*$ 
    on $\kall$, while the bottom one---the
    dependence of the estimate \eqref{eq:sigma_est} 
    of ${\sigma}_0^2$ on $\kall$.}
    \label{fig:k_alpha_sigma_joint}
    \vspace{-4pt}  
\end{figure}
There are several takeaways that we would like to highlight. In the theoretical part, we proved that provided $\kall$ is large than the given threshold, it is possible to recover $\pi^*$ and to estimate $k^*$. The threshold provided in theorems is rate optimal, however, it is not sharp. We observe that for different tasks the sharp threshold might differ up to a constant factor. For instance, in the case of known $k^*$ and $\s_0^2$ we see that the threshold seems close to $7$ (see Fig. \ref{fig:theorem_1_prec}), while it is around $32$ in the more difficult setting of unknown $k^*$ and $\s_0^2$, as it is shown in Fig. \ref{fig:k_alpha_sigma_joint}. The setting of Fig. \ref{fig:k_alpha} is somewhat in between, since $\s_0^2$ is assumed to be known but $k^*$ is unknown, hence the ``change point'' occurs around $22$, which is between $7$ and $32$. In practice, one could replace the theoretical quantiles of $\chi^2_d$ distribution with the empirical counterparts of the squared distances between matching pairs $X_i$ and $X_{\widehat{\pi}(i)}$ for a given estimator $\widehat{\pi}$. 

\subsection{Estimation of \texorpdfstring{$k^*$}{} for Real Data}
\label{ssec:RealData}

In what follows we carry out experiments on real data. In this experiment, we only focus on the estimation of the matching size in the setting of matching the keypoints (SIFT descriptors \citep{lowe2004distinctive}) of two different images of the same scene. The goal is to be able to match the same keypoints considering the noise present in the images, and not match the keypoints that are present in only one of the two images. We refer to Appendix B for more description and additional experiments in the keypoint matching problem.  

Experiments on real data were conducted using IMC-PT 2020 dataset \citep{Jin2020} consisting of 16 scenes with corresponding image sets and 3D point clouds. We used images from the ``Reichstag'' scene to illustrate how the procedure from \Cref{ssec:joint} can be used to estimate the matching size. To construct the dataset, we randomly chose $1000$ distinct image pairs of the same scene. Then, a scene point cloud is used to obtain pseudo-ground-truth matching between keypoints on different images of the same scene. Using these pseudo-ground-truth matching keypoints we then choose $k^*=60$ keypoints that are present in both images and add to each of them $40$ outlier keypoints. Hence, we have $n=m=100$ with $k^*=60$ and $d=128$-dimensional SIFT descriptors. Notice that in the case of images, no information on  $\sigma$ and $\sdiese$ is available. Moreover, noise levels might not be homoscedastic. 
In Fig.~\ref{fig:real_data_k}, we present the histogram of the estimated value $\widehat{k}$ and observe that the procedure described in \Cref{ssec:joint} provides a reasonably good estimate of $k^*=60$.

Recall that at each iteration of estimating $k^*$ we solve an MCF problem, which can be computationally costly given the sample sizes are large. However, there are several aspects where we can speed up this procedure. First, using the Greedy algorithm to match the feature vectors (as done in OpenCV) will allow us to compute only one distance per iteration instead of solving MCF from scratch. Second, the stepsize from the second step of the procedure from \Cref{ssec:joint} can be increased by considering the difference $\hat{\Phi}(k+10) - \hat{\Phi}(k)$, then with a proper adjustment to the threshold, we will obtain a $10$ times speedup in the number of iterations.

\begin{figure}
\centering
\includegraphics[width=0.47\textwidth]{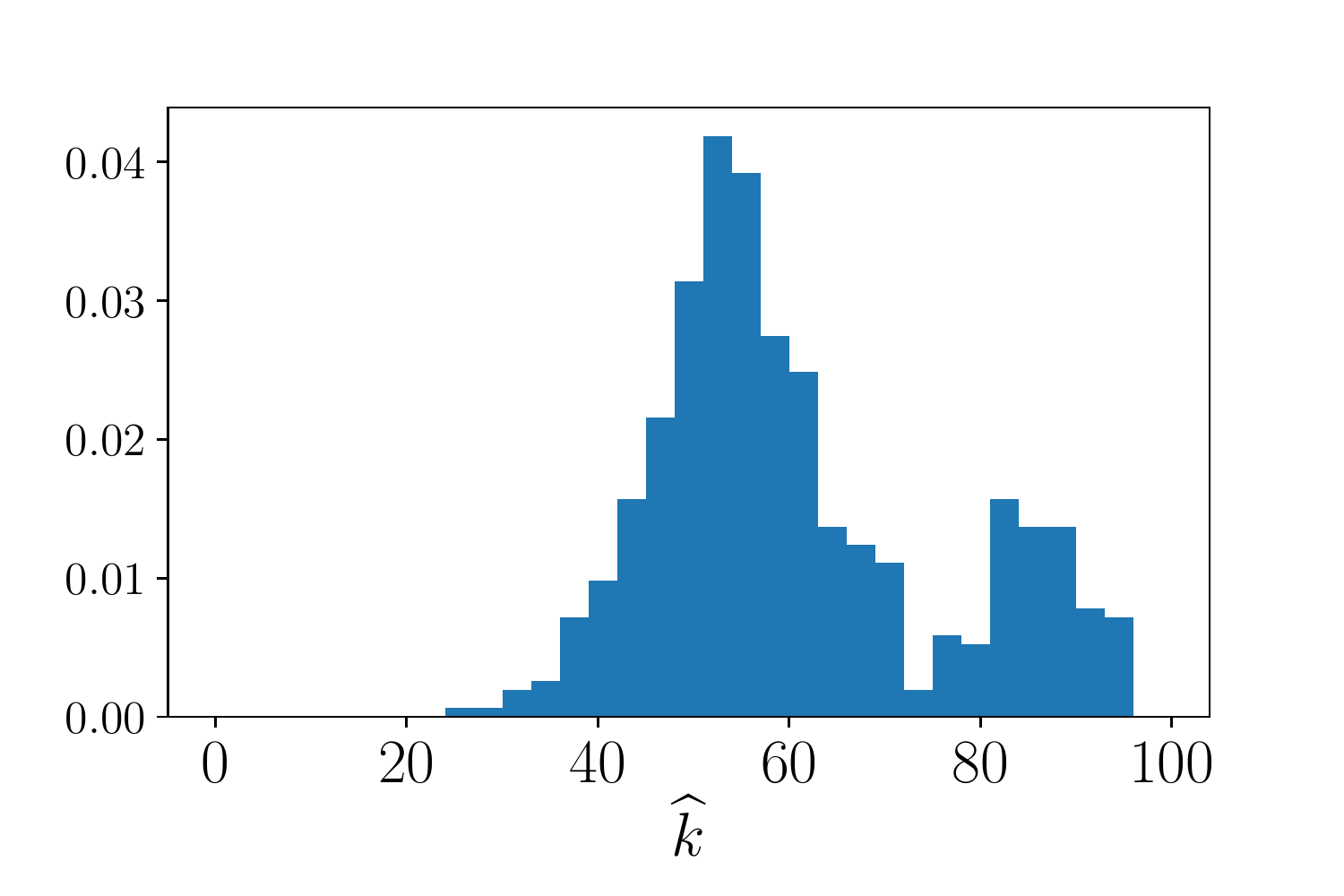}
\vspace{-10pt}
\captionof{figure}{The histogram of estimates values $\hat k$ computed on $1000$ distinct image pairs. In each image, a set of keypoints of sizes $n=m=100$ were chosen, with only $k^*=60$ inlier keypoints. Each keypoint is represented by its $128$-dimensional SIFT descriptor.}
  \label{fig:real_data_k}
\end{figure}

\subsection{Experiments on Biomedical Data}
\label{ssec:bio_data}
First, we tested our estimator $\widehat{\pi}_{k}^{\textup{LSS}}$ 
from \eqref{eq:lss-k} in the setting of \citep{chen2022one}[Sec. 5.1]. The setting considered there is to recover the matching between two datasets\footnote{\texttt{SeuratData} R package \citep{seurat}.} 
collected from human pancreatic islets using technologies (CEL-seq2 \citep{hashimshony2016cel} 
and Smart-seq2 \citep{picelli2013smart}). 
CEL-seq2 data contain measurements on 34363 RNAs in 2285 cells, 
and Smart-seq2 data contain measurements on 34363 RNAs in 
2394 cells. After applying standard pre-processing procedures 
using Python package \texttt{scanpy} \citep{wolf2018scanpy}, 
we select $5000$ most active RNAs for each dataset. $2808$ distinct 
RNAs appeared in both datasets' top-$5000$, so we leave out the 
rest obtaining two datasets of sizes $2808\times 2285$ and $2808\times 2394$. 
Each cell in both datasets has a human-annotated type (out 
of 13 cell types). We randomly downsample cells to get an equal number of cells-per-type in both datasets, eventually getting two datasets of size $2808 \times 1935$. We  match cells in two datasets using $\widehat{\pi}_{k}^{\textup{LSS}}$ from \eqref{eq:lss-k}. However, we have no information on the ground truth matching, therefore as suggested in \cite{chen2022one} the accuracy is calculated on the cell-type level. This means that a single match is considered correct if it matches two cells of the same type. Our matching estimator achieves $97.88\%$ cell-type level accuracy, which is almost the same as the result 
of $97.93\%$ of \citep{chen2022one} without any dimensionality 
reduction and using a simpler approach. The resulting 
confusion matrix is shown in Fig.~\ref{fig:celseq_smartseq_cm}.

\begin{figure}
\centering
\includegraphics[width=0.408\textwidth]{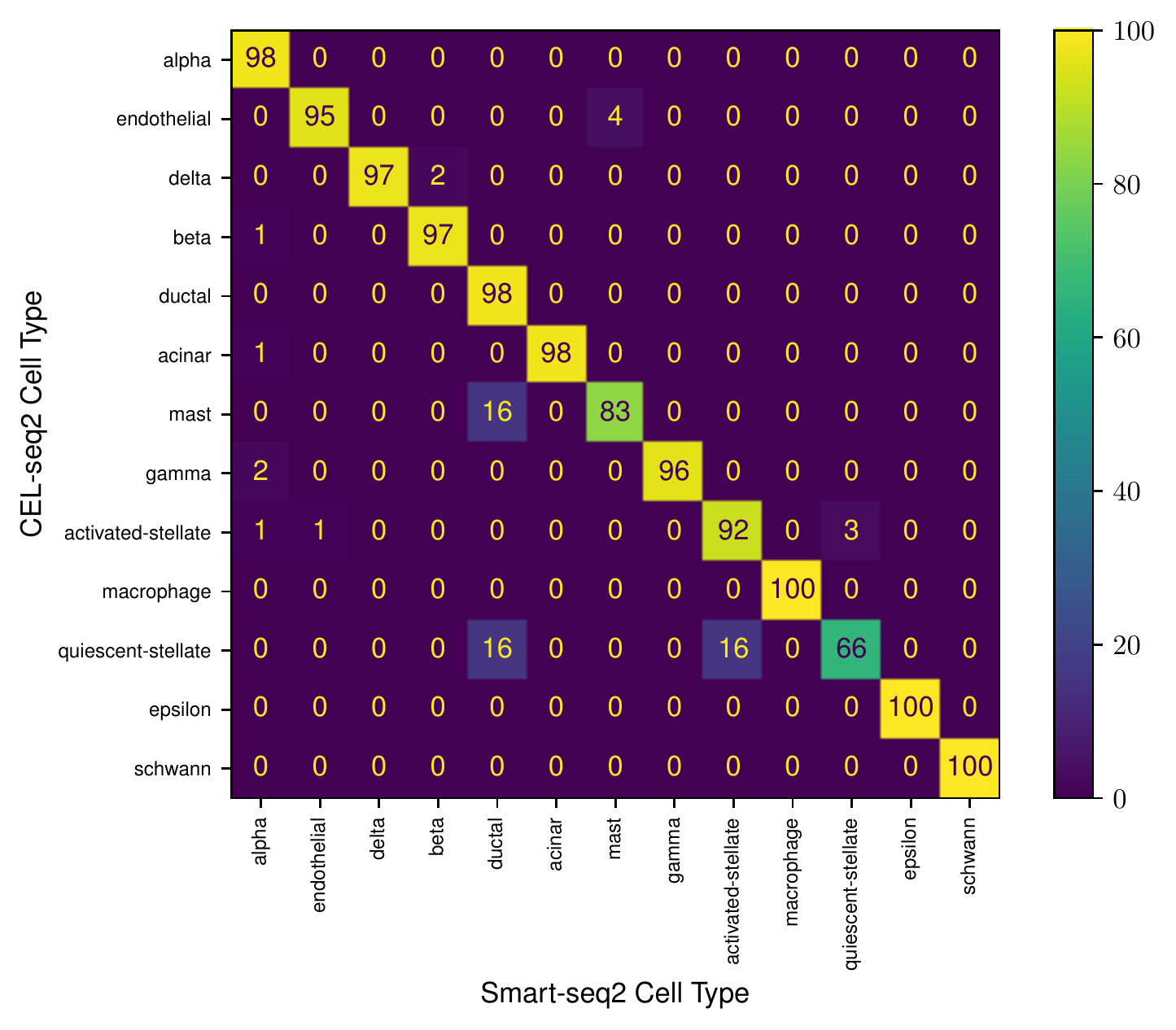}
\captionof{figure}{Cell-type level confusion matrix of $\widehat{\pi}_{n}^{\textup{LSS}}$ on Cel-seq2 and Smart-seq2 datasets of the same size $n=1935$.}
  \label{fig:celseq_smartseq_cm}
  \vspace{-15pt}
\end{figure}

To demonstrate the result of \Cref{thm:main}, we design 
the following experiment. After performing the same pre-processing 
and balancing steps from the previous experiment, we proceed to 
remove cells of one fixed type from the CEL-seq2 dataset and cells of 
another fixed type from the Smart-seq2 dataset. This way, we can 
ensure the presence of outliers in both sets. Number of cells-per-type 
are shown in Fig.~\ref{fig:celseq_smartseq_celltype_hist}. We only 
experiment with types \textit{beta, gamma, delta, ductal}, and 
\textit{acinar} because cells of type \textit{alpha} constitute 
almost half of the whole dataset and other types have very few cells for the results to be significant. Then we fix a pair of different cell types, remove each from the corresponding dataset, and estimate the real number of inliers (here meaning the number of cells of a type that appears in both datasets) using the algorithm described in \Cref{sec:theoreticalResults}. The results are reported in \Cref{tab:joint_est}. We observe that in most cases the estimated value $\widehat{k}$ is slightly underestimated but is an extremely accurate estimate of the true matching size $k^*$. 

\begin{figure}[ht]
\centering
\includegraphics[width=0.43\textwidth]{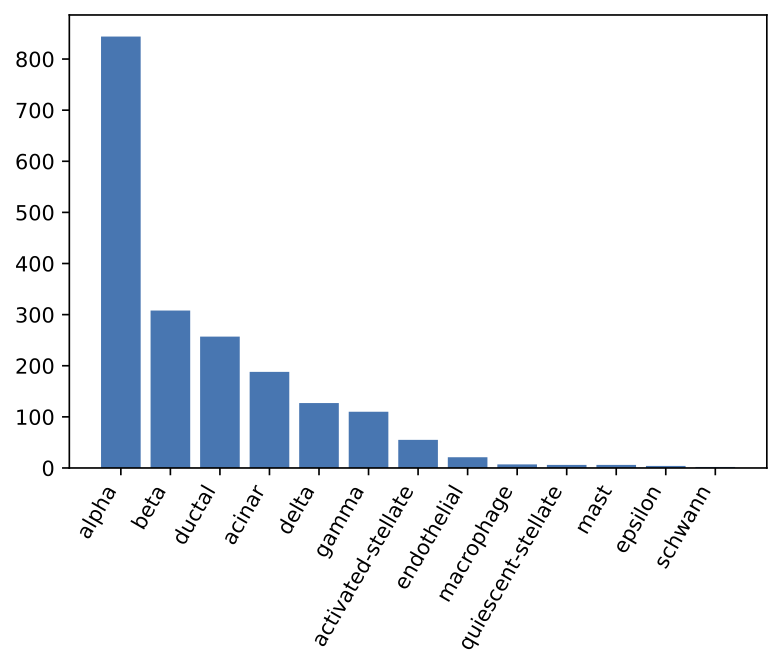}
\vspace{-5pt}
\captionof{figure}{Cell-type frequencies after type-balancing Cel-seq2 and Smart-seq2 datasets of size $n=1935$ with 13 cell-types.}
  \label{fig:celseq_smartseq_celltype_hist}
  \vspace{-5pt}
\end{figure}

\begin{table}[ht]
\scalebox{0.84}{
\begin{tabular}{|l|l|l|l|l|l|
}
\hline
       & $\beta$ & $\gamma$ & $\delta$ & duct. & aci. 
       \\ 
\hline
$\beta$   & ~  &\nicefrac{1517~\, }{~\, 1593} &\nicefrac{1500~\, }{~\, 1592}&\nicefrac{1370~\, }{~\, 1362}&\nicefrac{1439~\, }{~\, 1421}
\\ \hline
$\gamma$  & \nicefrac{1517~\, }{~\, 1594} & ~   &\nicefrac{1698~\, }{~\, 1777}&\nicefrac{1568~\, }{~\, 1557} & \nicefrac{1637~\, }{~\, 1618}
\\ \hline
$\delta$  & \nicefrac{1500~\, }{~\, 1594} &\nicefrac{1698~\, }{~\, 1772}& ~ & \nicefrac{1551~\, }{~\, 1540} & \nicefrac{1620~\, }{~\, 1602}
\\ \hline
duct. & \nicefrac{1370~\, }{~\, 1362} &\nicefrac{1568~\, }{~\, 1559}&\nicefrac{1551~\, }{~\, 1544} & ~   & \nicefrac{1490~\, }{~\, 1557}
\\ \hline
aci.  & \nicefrac{1439~\, }{~\, 1434} &\nicefrac{1637~\, }{~\, 1631}&\nicefrac{1620~\, }{~\, 1614}    & \nicefrac{1490~\, }{~\, 1592} & ~    
\\ \hline
\end{tabular}
}
\caption{
Each row corresponds to the experiment where the particular type has been removed from the CEL-seq2 dataset and 
each column---to the type removed 
from Smart-seq2. Each cell has two numbers;  
the true number of inliers / its estimate by the 
algorithm described in \Cref{sec:theoreticalResults}.}
\label{tab:joint_est}
\end{table}
\section{CONCLUSION AND DISCUSSION}
\label{sec:discussion}
We have analyzed the problem of matching map recovery between two sets of feature-vectors, when the number $k^*$ of true matches is unknown. We focused on two practically relevant settings of this problem. Assuming a lower bound $k$ on $k^*$ is available, we proved---under the weakest possible condition on the signal-to-noise ratio---that the $k$-LSS procedure makes no mistake with high probability. More precisely, 
$k$-LSS provides an estimated map the support of which is 
included in the support of the true matching map and 
the values of these two maps coincide on this subset. 
More importantly, we proposed a procedure for estimating 
the unknown matching size $k^*$ and proved that it finds 
the correct value of $k^*$ and the true matching map 
$\pi^*$ with high probability. Once again, this holds 
under the minimal assumption that the signal-to-noise 
ratio exceeds the minimax separation rate. 

Interestingly, our results demonstrate that the minimax 
rate of separation does not depend on $k ^*$ and, more 
surprisingly, that the absence of the knowledge of $k^*$ 
has no impact on the minimax rate. These rates are attained
by computationally tractable algorithms solving the minimum cost flow problem. Our results are limited to Gaussian noise and to noise 
levels that are equal across observations. Furthermore, 
we only tackled the recovery problem, leaving the problem
of estimation to future work. 
\subsubsection*{Acknowledgements}
This work was supported by the grant Investissements d’Avenir (ANR-11-IDEX0003/Labex Ecodec/ANR-11-LABX-0047), the ADVANCE Research Grant provided by the Foundation for Armenian Science and Technology, and the Yerevan State University.


\bibliography{permutation_ref}

\onecolumn

\appendix

\part{Appendix} 
\parttoc 

The purpose of this appendix is twofold: to present the proofs of the lemmas used in the main paper and to provide additional experimental evidence showing that it is indeed possible to obtain an accurate estimator for unknown $k^*$ even when the noise levels are unknown and potentially heterogeneous. The reproducible code of all the experiments can be found in the supplementary material. 

\section{Proofs of Lemmas from \texorpdfstring{\Cref{sec:main_proof}}{}}\label{app-A}
We start by presenting the definitions that we use in this supplementary material. Recall the definitions of the test statistics $\hat{\Phi}(\cdot)$ and its normalized version $\hat{L}(\cdot)$, which depends on the quantity $\sigma_0^2$
\begin{align}
    \hat \Phi(k) = 
     \min_{\pi \in \mathcal{P}_k} 
     \sum_{i \in S_{\pi}}
     {\| X_i - \Xdiese_{\pi(i)} \|^2_2}, \qquad 
     \hat L(k) = \min_{\pi \in \mathcal{P}_k}
     \sum_{i \in S_{\pi}} 
     \frac{\| X_i - \Xdiese_{\pi(i)} \|^2_2}
     {\sigma^2 + \sigmadiese{}^2} \equiv \frac{\hat \Phi(k)}{\sigma_0^2}.
\end{align}
For completeness, we also recall the definition of the standard Gaussian random vectors $\eta_{ij}$
\begin{align}
    \eta_{ij} = \frac{\sigma\xi_i - \sigmadiese\xidiese_j}{\sqrt{\sigma^2+{\sigmadiese}{}^2}}.
\end{align}

The quantities associated with $\eta_{ij}$ which will be used in the proofs are $\zeta_1$ and $\zeta_2$, which are defined as follows
\begin{align}
    \zeta_1 \triangleq \max_{i \neq j} \frac{|(\theta_i - \tdiese_j)^\top\eta_{ij}|}{\| \theta_i - \tdiese_j\|_2}, \qquad \zeta_2 \triangleq d^{-1/2} \max_{i, j} \big|\|\eta_{ij}\|^2_2 - d \big|.
\end{align}
Recall also that for any matching map $\pi$ we define $L(\pi)$ as follows
\begin{align}
    L(\pi) = \sum_{i \in S_{\pi}}\nolimits \frac{\| 
    X_i - \Xdiese_{\pi(i)}\|_2^2}{\sigma^2 + \sdiese{}^2}.
\end{align}

In this section, we present the proofs of lemmas used in \Cref{sec:theoreticalResults} for proving \Cref{theorem_2} and \Cref{thm:3}. For the reader's convenience, we include the statements of the lemmas as well.

\lemmauno*
\begin{proof}[Proof of \Cref{lem:1}] 
Let us recall the definition of the individual signal-to-noise ratios $
\k_{i,j} \triangleq \frac{\| \theta_i - \tdiese_j\|_2}{\sqrt{\sigma^2+{\sigmadiese}{}^2}}$. 
For any matching map $\pi$ and 
for any $i\in S_\pi$, we have
\begin{align}
     \frac{\| X_i - \Xdiese_{\pi(i)} \|^2_2}
     {\sigma^2+{\sigmadiese}{}^2}& = \frac{\| \theta_i - \tdiese_{\pi(i)}\|^2_2}{\sigma^2+{\sigmadiese}{}^2} + 2 \frac{(\theta_i - \tdiese_{\pi(i)})^\top\eta_{i,\pi(i)}}{\sqrt{\sigma^2+{\sigmadiese}{}^2}} +  \|\eta_{i,\pi(i)}\|^2_2 \\
    & \geq  \frac{\| \theta_i - \tdiese_{\pi(i)}\|^2_2}{\sigma^2+{\sigmadiese}{}^2} - 
    2 \zeta_1 \frac{\| \theta_i - \tdiese_{\pi(i)}\|_2}{\sqrt{\sigma^2+{\sigmadiese}{}^2}} + \|\eta_{i,\pi(i)}\|^2_2\\
    & =  \k_{i, \pi(i)}^2 - 2 \zeta_1 \k_{i, \pi(i)} + \|\eta_{i,\pi(i)} \|_2^2. \label{eq:5}
\end{align} 
Note that if $i\in S_\pi$ is such that $\pi^*(i) = \pi(i)$ (correct matching), then $\k_{i, \pi(i)} = 0$. For all the other $i\in S_\pi$, we have
$\k_{i, \pi(i)} \ge \kall$. Therefore, denoting 
$S_\pi^+ = \{i\in S_\pi\cap S_{\pi^*} : \pi(i) =
\pi^*(i)\}$ and $S_\pi^- = S_\pi\setminus S_\pi^+$, 
Eq.~\eqref{eq:5} implies that on the event 
$\{ \kall \ge \zeta_1\}\supset \O_0$, 
we have
\begin{align}
    L(\pi) 
    & \geq  \sum_{i \in S_{\pi}^-} \big( \k_{i, \pi(i)}^2 - 2 \zeta_1 \k_{i, \pi(i)}\big) + \sum_{i \in S_{\pi}}\|\eta_{i,\pi(i)} \|_2^2   \\
    & \geq |S_\pi^-| \big( \kall^2 - 2 \zeta_1 \kall \big) + \sum_{i \in S_{\pi}} \|\eta_{i,\pi(i)} \|_2^2 . 
\end{align} 
Let us choose any $S_0$ such that 
$S_\pi^+\subset S_0\subset S^*$ and $|S_0|\le |S_\pi|$. We define $\pi_0$ as the restriction 
of $\pi^*$ on $S_0$.  
If, in addition, we set $S_0^- = S_0 
\setminus S_{\pi}^+$, we can infer from the last
display that 
\begin{align}
    L(\pi) - L (\pi_0) &\ge  |S_\pi^-| \big( \kall^2 - 2 \zeta_1 \kall \big) + \sum_{i \in S_{\pi}^-} \|\eta_{i,\pi(i)} \|_2^2
    -\sum_{i \in S_0^-} \|\eta_{i,\pi^*(i)} \|_2^2\\
    &\ge |S_\pi^-| \big( \kall^2 - 2 \zeta_1 \kall
    \big) + d(|S_\pi^-| - 
    |S_0^-|) - \sqrt{d}\,\zeta_2(|S_\pi^-| + 
    |S_0^-|)\\
    &\ge |S_\pi^-| \big( \kall^2 - 2 \zeta_1 \kall - 2 \sqrt{d}\,\zeta_2\big) + d(|S_\pi^-| - 
    |S_0^-|)\\
    &=|S_\pi^-| \big( \kall^2 - 2 \zeta_1 \kall - 2 \sqrt{d}\,\zeta_2\big) + d(|S_\pi| - 
    |S_0|). \label{eq:6}
\end{align}
On the event $\Omega_0$, we have $\kall^2 - 2 \zeta_1 \kall - 2 \sqrt{d}\,\zeta_2\ge 
\kall^2/4$. Moreover, since $\pi\neq \pi^*_{S_\pi}$, we have $|S_\pi^-|\ge 1$. These two inequalities 
combined with \eqref{eq:6} complete the proof of
the lemma.
\end{proof}

\lemmados*

\begin{proof}[Proof of \Cref{lem:2}]
The union bound implies that 
\begin{align}
    \prob(\O_{0, x}^\complement)
        &\le \prob\big(8\z_1\ge x\big)
            +\prob(4\sqrt{d}\,\z_2 \ge x^2)\\
        &= \prob\big(\z_1\ge {\textstyle\frac1{8}}\, {x}\big)+
            \prob\bigg(\z_2\ge{\textstyle \frac{1}{4\sqrt{d}}}\,x^2\bigg).
            \label{ineq:zeta1zeta2}
\end{align}
Notice that $\zeta_1$ can be represented as the maximum of absolute values of standard Gaussian random variables, \textit{i.e.}, $\zeta_1 = \max_{i \neq j} |\zeta_{i, j}|$. Applying the well-known Gaussian tail bounds together with the union bound
yields
\begin{align}
    \prob\big(\z_1\ge {\textstyle\frac1{8}}\, {x}\big) \le \sum_{i \neq j} \prob\big(|\zeta_{i, j}| \ge {\textstyle\frac1{8}}\, {x} \big) \le 2n^2 \exp\big(-x^2/128\big). 
    \label{ineq:zeta1_}
\end{align}

To bound the second term of \eqref{ineq:zeta1zeta2}, we use Lemma 1 from \cite{galstyan2021optimal} which bounds the tails of a random variable $\zeta_2$. Thus, combining it with a union bound we arrive at the following inequality
\begin{align}
    \prob\bigg(\z_2\ge \frac{x^2}{4\sqrt{d}}\bigg) &\le 2n^2 \exp\bigg\{ - \frac{x^2}{32\sqrt{d}} \bigg(\frac{x^2}{4\sqrt{d}} \wedge \sqrt{d}\bigg) \bigg\} \\
    &= 2n^2 \exp\bigg\{ - \frac{(x/16)^2}{d}\big(2x^2 \wedge 8d\big) \bigg\}. 
    \label{ineq:zeta2_}
\end{align}
Then, plugging the bounds obtained in \eqref{ineq:zeta1_} and \eqref{ineq:zeta2_} into \eqref{ineq:zeta1zeta2} concludes the proof of the lemma.
\end{proof}

\lemmatres*

\begin{proof}[Proof of \Cref{lem:5}]
This claim is a consequence of \Cref{lem:1}. 
We have already seen in the proof of \Cref{theorem_2} that $\hat\pi_{k^*} = \pi^*$ 
on $\O_0$. Therefore, 
\begin{align}
    \hat L(k^*+1) - \hat L(k^*) = 
    L(\hat\pi_{k^*+1}) - L(\pi^*). 
\end{align}
If we apply \Cref{lem:1} to $\pi = \hat\pi_{k^*+1}$, it is clear that we can choose as $\pi_0$ the true 
matching map $\pi^*$. The claim of \Cref{lem:1} then
yields
\begin{align}
    L(\hat\pi_{k^*+1}) - L(\pi^*) \ge \frac14\kall^2
    + d(k^*+1 - k^*) = \frac14\kall^2
    + d
\end{align}
and the claim of the lemma follows.
\end{proof}

\lemmaquatro*

\begin{proof}[Proof of \Cref{lem:6}] 
Let $\hat\pi_k$ be a matching map from 
$\mathcal P_k$ minimizing $L(\cdot)$, \textit{i.e.}, such that $L(\hat\pi_k) = \hat L(k)$. According to
\Cref{lem:1}, we have $\hat\pi_k(i) = \pi^*(i)$ 
for every $i\in \hat S_k \triangleq S_{\hat\pi_k}$. One easily checks that there exists a set 
$\hat S_{k+1}\subset S^*$ of cardinality $k+1$ 
such that $\hat S_k\subset \hat S_{k+1}$ and 
$\hat L(k+1) = L(\hat\pi_{k+1})$ where 
$\hat\pi_{k+1}$ is the restriction of $\pi^*$ 
to $\hat S_{k+1}$. Indeed, if $\pi$ is
any element of $\mathcal P_{k+1}$ minimizing $L(\cdot)$, we know that it is defined as a restriction of $\pi^*$ on some set $S$ of cardinality $k+1$. If we replace arbitrary $k$ 
elements of $S$ by those of $\hat S_k$, and 
modify $\pi$ accordingly, then we will get
a new mapping from $\mathcal P_{k+1}$, for
which the value of $L(\cdot)$ is less than
or equal to $L(\pi)$. Therefore, we have found
a mapping map that minimizes $L(\cdot)$ over
$\mathcal P_{k+1}$ and has a support that is
obtained by adding one point to $\hat S_k$. 
This implies that
\begin{align}
    \hat L(k+1) - \hat L(k)  = 
    L(\hat\pi_{k+1}) - L(\hat\pi_k)
    & = \sum_{i\in\hat S_{k+1}} \|\eta_{i,\pi^*(i)}\|^2_2 - 
    \sum_{i\in\hat S_{k}} \|\eta_{i,\pi^*(i)}\|^2_2\\
    & = \sum_{i\in\hat S_{k+1}\setminus \hat S_k} \|\eta_{i,\pi^*(i)}\|^2_2 \le d + \sqrt{d}\,\zeta_2.
\end{align}
This completes the proof of the lemma.
\end{proof}

\section{Additional experiments on real data}\label{app-B}
In this section, we perform experiments on a pair of images and respectively choose keypoints on each of them to showcase the behavior of the proposed procedure. In Fig.~\ref{fig:real_data_k} of the main manuscript, we show the histogram of the choice of matching size $k^*$ for $1000$ distinct image pairs from ``Reichstag'' scene of IMC-PT 2020 dataset \citep{Jin2020}. In this dataset, we only have the {\it{pseudo}} ground truths, and sometimes these ground truths are incorrect (different points are matched), which makes the results unreliable. Therefore, to have a more controlled experiment we take one image of Sacr\'{e} Coeur of Paris and crop it in half on each axis. Then, we add noise into the cropped image by interpolating the pixels such that both images have the same resolution. This procedure is in line with the studied model, presented in \eqref{model}. Afterward, we detect and compute SIFT descriptors of $m = 2n - k^*$ keypoints from the cropped image and translate them into the original image. Then, we fix $k^*$ {\it{inlier}} keypoints in both images and add $n-k^*$ distinct points to each image, which will be considered as \textit{outliers}.

\begin{figure}[!ht]
    \centering
\includegraphics[scale=0.35]{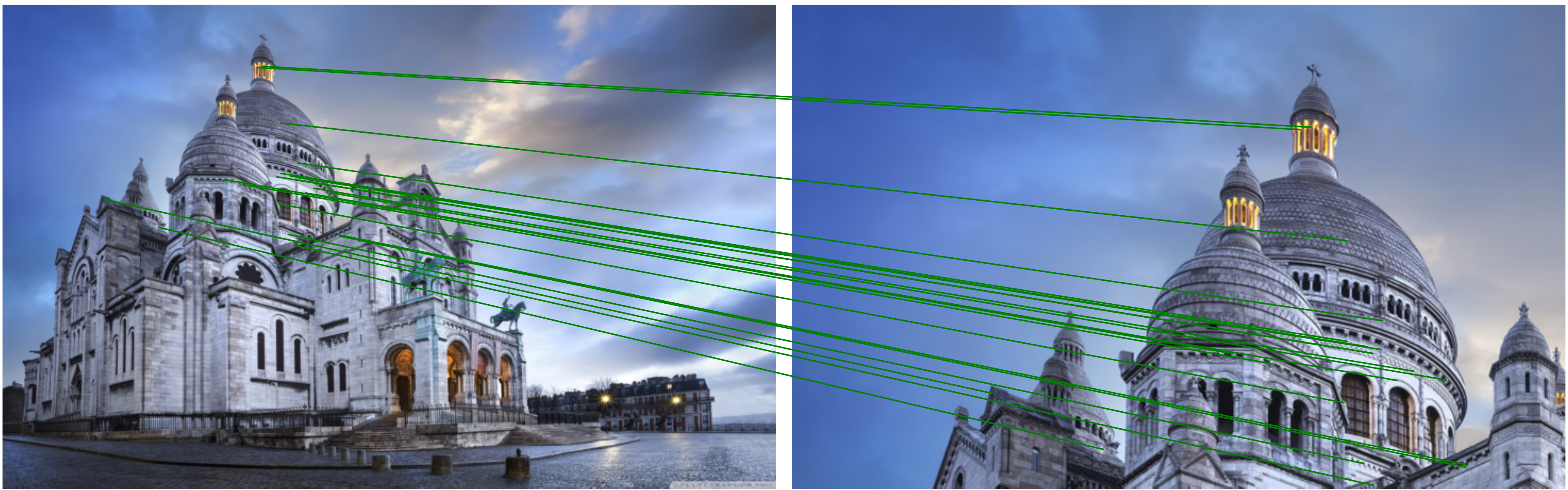}
\includegraphics[scale=0.35]{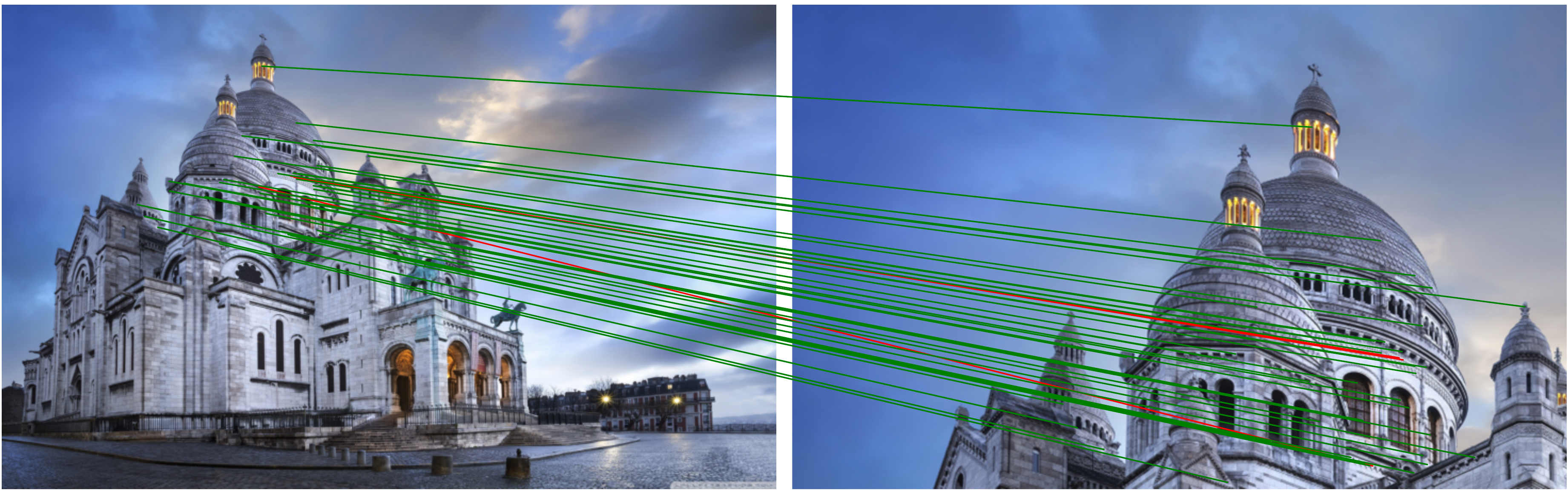}
\includegraphics[scale=0.35]{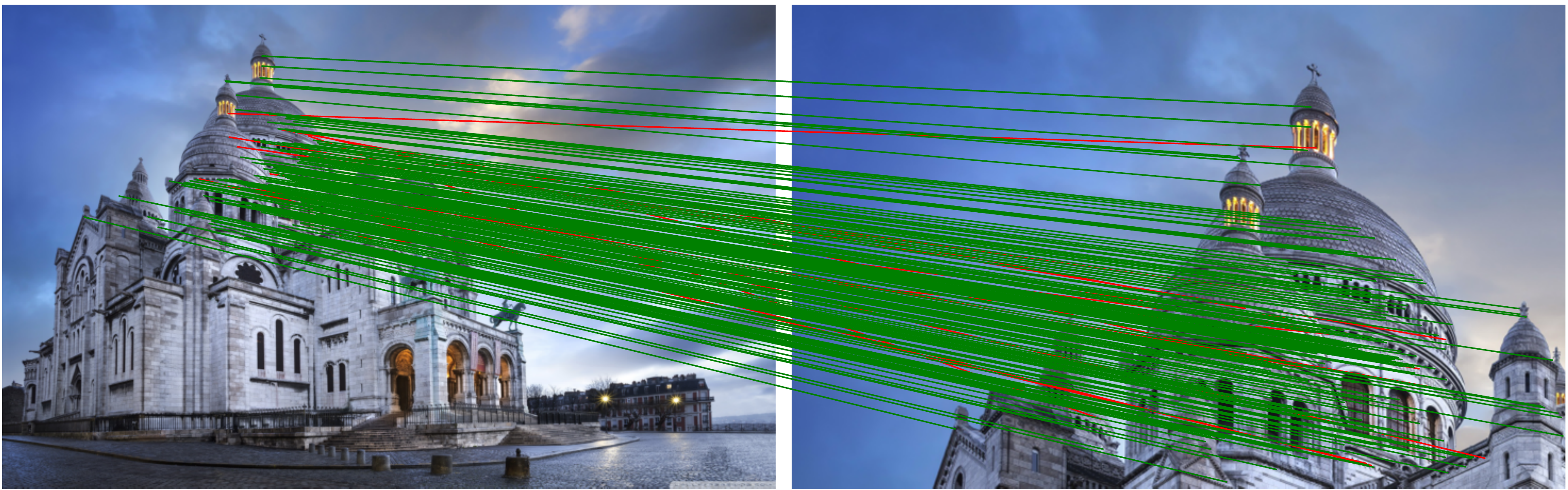}
\caption{
    We fixed the tolerance level $\alpha = 0.001$. For illustration purposes, we have selected 3 different scenarios to demonstrate the quality of the model selection and the matching recovery. In the uppermost plot $n=50$ and $k^*=25$. The procedure from \Cref{sec:theoreticalResults} outputs $\hat{k} = 22$ and $\hat{\pi}_{\hat k}^{\textup{LSS}} = \pi^*$ (perfect matching). In the middle plot $n=100$ and $k^*=50$, the estimated value of $k^*$ is $\hat{k} = 44$. In the bottom plot we selected $n=350$ keypoints from which $k^*=250$ were inliers ($\hat{k} = 213$). In last two cases matching map contained only few mistakes.
}
\label{fig:scoeur}
\end{figure}

We then run our procedure for the estimation of the matching size $k^*$ and the recovery of the matching map $\pi^*$ with the tuning parameters chosen as shown in \Cref{thm:main}. The results for different values of $n$ and $k^*$ are summarized in the figure below. The estimated value $\hat{k}$ is close to $k^*$ and is slightly underestimated in all cases. Slight underestimation is not a problem, whereas slight overestimation would surely cause more incorrect matching pairs.

For all the plots, we see that the value of $k^*$ is estimated accurately and is slightly underestimated (as shown also in Fig.~\ref{fig:real_data_k}). The accuracy of the estimation of $\pi^*$ (number of green lines) is also very high with only a few mistakes. It is worth mentioning that the estimation of $k^*$ is a procedure that can be of interest by itself because after having an accurate estimator for $k^*$ one is free to apply any matching algorithm to circumvent other purposes. For example, one can use fast approximate methods to accelerate matching algorithms (see e.g. \cite{malkov2020, harwood2016, jiang2016}). Another possible direction is to consider wider or narrower classes of mappings, e.g. 1-to-many matching maps.

\end{document}